\documentclass[12pt]{amsart}

\usepackage[usenames]{color}
\input xypic

\usepackage{verbatim}
\usepackage{url}
\usepackage{bm}
\usepackage{pict2e}

\usepackage{amsmath}
\usepackage[centertags]{amsmath}
\usepackage{amsfonts}
\usepackage{amssymb}
\usepackage{amsthm}
\usepackage{newlfont}
\usepackage{url}
\usepackage{bigints}

\theoremstyle{plain}
\newtheorem{thm}{Theorem}

\newtheorem{lem}[thm]{Lemma}
\newtheorem{lem*}[thm]{Lemma}
\newtheorem{prop}[thm]{Proposition}

\theoremstyle{definition}
\newtheorem{dfn}{Definition}
\theoremstyle{remark}
\newtheorem{rem}{Remark}
\newtheorem{rem*}{Remark}
\newtheorem{ex}[rem]{Example}
\newtheorem{example}[rem]{Example}

\numberwithin{rem}{section} 
\numberwithin{dfn}{section} 
\numberwithin{equation}{section} 
\numberwithin{thm}{section} 

\def\!{\operatorname{!}}

\def\be{\begin{equation}}
\def\ee{\end{equation}}
\def\beg{\begin{equation*}}
\def\eeg{\end{equation*}}
\def\sgn{\mathrm{sgn}}

\def\B{\mathbb B}

\def\1{\bold 1}

\begin{document}

\title{New approach to certain real hyper-elliptic integrals}

\author{Piotr Kraso\'n, Jan Milewski}
\begin{abstract}
In this paper we treat certain elliptic and hyper-elliptic integrals in a unified way. We introduce a new  basis of these  integrals coming from certain basis ${\phi}_n(x)$ of polynomials and show that the transition matrix between this basis and the traditional monomial basis is certain upper triangular band matrix. This allows us to obtain explicit formulas for the considered integrals. Our approach, specified to elliptic case, is more effective than known recursive procedures for elliptic integrals.
We also show that  basic integrals  enjoy symmetry coming from the action of the dihedral group $D_{M}$ on a real projective line. This action is closely connected with the properties of homographic transformation of a real projective line. This explains similarities occurring 
 in some formulas in popular tables of elliptic integrals. As a consequence one can  reduce the number of necessary formulas in a significant way. 
 We believe that our results will simplify programming and computing the hyper-elliptic integrals in various problems of mathematical physics and engineering.
\end{abstract}
\date{\today}

\address{  Institute of Mathematics, Department of Mathematics and Physics, University of Szczecin, ul. Wielkopolska 15, 70-451 Szczecin, Poland 
}
\email{ piotrkras26@gmail.com}

\address{Institute of Mathematics, Faculty of Electrical Engineering, Pozna{\'n} University of Technology, ul. Piotrowo 3A, 60-965 Poznań, Poland}
\email{jsmilew@wp.pl}

\thanks{}

\maketitle


\section{Introduction} 
 
The aim of this paper is to give explicit formulas for the integrals of the following form
\begin{equation}\label{int1}
I_{n,p}={\int}{\frac{(x-p)^ndx}{\sqrt{Q(x)}}}, \qquad n\in{\mathbb Z},
\end{equation}
where $Q(x)$ is a polynomial of degree $M$  with $M$ real  zeroes of multiplicity 1.   It is very well known that when $M>2$ the integrals are in general  not elementary. The cases $M=3$ and $M=4$ lead to elliptic integrals, whereas 
for $M\ge 5$ we obtain hyper-elliptic integrals. The computation of the elliptic integrals in general, and of type (\ref{int1}) in particular,  is based on a recursive formula for these integrals. In this way all elliptic integrals of the form (\ref{int1}) can be expressed as  linear combinations of fundamental  integrals  (\ref{uklad})    and  elementary functions. The coefficients at the  fundamental integrals  are rational functions (cf. \cite{kk00}, \cite{ps97})).
Nevertheless, the usage of this recursive formula is cumbersome. In this paper we offer a new method for computation of the integrals of type ( \ref{int1}) for arbitrary $M$ and $n.$ 
We give  explicit formulas for this computation. We choose a  special basis in the vector space generated by the family of monomials $(x-p)^n$, $n\in {\mathbb Z}$. 
As the fundamental integrals  we choose $\{I_{-1,p},\,\, p\in{\mathbb R}, \,\, Q(p)\neq 0\}, I_{0}, \dots  I_{M-2}.$  For shorthand,    we denote the integrals $I_{i,0}$ as  $I_{i}, \quad i=0,\dots ,M-2.$
The choice of the basis allows one to compute all the integrals in the family (\ref{int1}) as   sums 
\beg
 \int \frac{x^n dx}{\sqrt{Q(x)}}=\sum_{l=0}^{M-2}B_{l,n}I_l+2\sum_{l=M-1}^n B_{l,n} x^{l+1-M}\sqrt{Q(x)}+C \eeg
for $n>M-2$ and
\beg
 \int (x-p)^n \frac{ dx}{\sqrt{Q(x)}}=\sum_{l=-1}^{M-2}U_{l,n}I_{l,p}+2\sum_{l=n}^{-2} U_{l,n} (x-p)^{{l+1}}\sqrt{Q(x)}+C \eeg
 for $n<-1.$
 
Notice that  the coefficients $B_{l,n}$ (resp.  $U_{l,n}$ )  constitute the $n$-th column of the upper triangular matrix $B$ (resp. U ) and can be found by an easy  recurrence.
This recursive procedure in the elliptic case is much simpler than the original one (cf. \cite{ps97}). We illustrate this in  the examples \ref{ex73} and \ref{ex74}.
For  a general hyper-elliptic integral our  method is very clear and easy to apply.  It is also efficient for big values of $n$. 

Symmetries for elliptic integrals and elliptic functions are interesting from both theoretical and computational 
points of view (cf. \cite{c64}, \cite{c04}, \cite{c06},\cite{c10}).
 Analyzing various tables with the elliptic integrals (cf. \cite{as72}, \cite{bf71}, \cite{gr00} ) we see that there are many similar formulas depending on between which roots  there is a variable of integration $x.$
Using appropriate permutations of the  roots  which constitute the dihedral group $D_{4}$ we can obtain many formulas  from the suitable one.   To do this we use one point compactification of a real line. 
Observe  that this approach can  also be applied in a hyper-elliptic case where  we obtain an action of the dihedral group $D_M.$
 We show that this action comes from the action of the homographic transformations on $S^1\cong {\mathbb P}^1_{\mathbb R}$.

\section{Recurrence}
In this section we prove that any hyper-elliptic integral of type (\ref{int1}) can be expressed in terms of fundamental integrals. This is well known for the elliptic case (cf. \cite{ps97}), the hyper-elliptic 
case is   a straightforward  generalisation. We include it for completeness. 

It is clear that for our purposes  it is enough to consider the following integrals
\be \label{wystarczy} I_{n}=
 \bigintssss  \frac{x^ndx}{\sqrt{Q(x)}}, \,\,\qquad {\mathrm{for}}\,\,n\geq 0\ee
 and
\be\label{enough}
I_n= \bigintssss \frac{(x-p)^ndx}{\sqrt{Q(x)}} \,\,\qquad{\mathrm{for}}\,\,n<0.
 \ee
The key fact we use is that certain combinations of hyper-elliptic integrals are elementary functions.
Indeed, consider the derivative of the product $2u^{n+1} \sqrt{Q(x)}$, where $u=x-p$, 
\be \label{uc1} Q(x)=\sum_{j=0}^M a_j x^j=\sum_{j=0}^M b_j u^j, \quad b_j=b_j(p). \ee
We have
\be \label{uc2} (2u^{n+1} \sqrt{Q(x)})'=\sum_{j=0}^M[2(n+1)+j]b_ju^{j+n}/\sqrt{Q(x)} .\ee
Integrating the formula (\ref{uc2})  one obtains
\be \label{uc3} \sum_{j=0}^M [2(n+1)+j] b_j I_{n+j,p}(x)=2(x-p)^{n+1}  \sqrt{Q(x)} +{C}. \ee
The equality (\ref{uc3}) shows that the integrals $I_{n,p}$ for $n\geq 0$ can be recursively written 
using $I_0, I_{1,p},\dots , I_{M-2,p}$ and $I_{M-1,p}.$
Substituting  $n=-1$  into (\ref{uc2}) we obtain 
\beg \sum_{j=0}^{M-1} (j+1)b_{j+1} I_{j,p}(x)=2\sqrt{Q(x)}+C. \eeg\label{aux1}

Thus we can compute $I_{M-1,P}$ in terms of the integrals $I_{0}, I_{1,p},\dots , I_{M-2,p}$
\beg\label{auxeq}
I_{M-1,p}=\frac{1}{Mb_M}\left(2\sqrt{Q(x)}- \sum_{j=0}^{M-2} (j+1)b_{j+1} I_{j,p}(x)\right) .
\eeg
Hence all integrals of the form $I_{n,p}$ for $n\geq 0$ can be expressed by means of
$ I_0, I_{1,p} \ldots I_{M-2,p} .$
        Naturally, 
\be \label{newton} I_{n,p}=\sum_k (-1)^k {n \choose k} p^k I_{n-k}.\ee
So, for $n\ge 0,$  these integrals can be expressed by 
$ I_0, I_1, \ldots I_{M-2} . $
In order to obtain integrals $I_{n,p}$ for all integers $n,$ we see that  by (\ref{uc3}),  it is enough to add to the last system, 
the integrals of the form $ I_{-1,p}.$ Hence the  integrals $I_{n,p}$ can be expressed by means of
the  integrals:
\be\label{uklad} I_{-1,p}, I_0, I_1, \ldots I_{M-2} . \ee
In the case when $p$ is a root of $Q$ the integral $I_{M-2} $ can be expressed by means of the remaining   integrals $I_{-1,p},  I_1 \ldots I_{M-3}$. This is because $b_0=0$ and the formula (\ref{uc2}) has one less nonzero terms cf. Example \ref{ex21}.  Equivalently,
the integral $I_{-1,p}$ can be expressed by $ I_1, \ldots I_{M-2}$.
This shows that one can take, as basic integrals, the following set
$ \{I_{-1,p}: \,\,p\in {\mathbb R} ,\,\, Q(p)\neq 0\}, I_0, I_1, \ldots I_{M-2}. $

 Thus we have proved the following:
\begin{prop}\label{prop1} The following integrals 
\beg\label{basis}  \{I_{-1,p}: \,\, p\in{\mathbb R}, \,\,Q(p)\neq 0\}, I_0, I_1, \ldots I_{M-2} .\eeg 
form a basis  for hyper-elliptic integrals  (\ref{int1}) i.e. any hyper-elliptic integral  can be expressed by a linear combination, with coefficients being rational functions, of the basic integrals and elementary functions.
\end{prop}
\qed

The following example illustrates this.

{\begin{example}\label{ex21} Let $Q=b_3x^3+b_2x^2+b_1x$. Hence $0$ is a root of $Q$ and
$$ \left( \frac{2}{x} \sqrt{Q(x)} \right)'=(b_3x-\frac{b_1}{x})/\sqrt{Q(x)} .$$
So, 
$ b_3I_1-b_1I_{-1}=\frac{2}{x} \sqrt{Q(x)}. $
 This shows that the integral $I_{-1}$ can be expressed by $I_1$ and vice versa.
In this case one can take as the basic integrals either 
$ \{I_{-1, p} : p\in {\mathbb R}
\}, I_0 $
or
$ \{I_{-1, p} : \,\, p\in {\mathbb R},\,\,Q(p)\neq 0\}, I_0, I_1  .$
\end{example}


\section{hyper-elliptic integrals for a polynomial factor.}

We put $p=0$ and consider $n\geq -1$ in (\ref{uc2})-(\ref{uc3}). 
 For this case $u=x, b_j=a_j$ and  shifting properly $n$ we have
\be \label{u0} (2x^{n+1-M} \sqrt{Q(x)})'=\sum_{l=n-M}^n(l+n-M+2) a_{l+M-n}x^{l}/\sqrt{Q(x)} , \quad n\geq M-1.\ee
Observe that for $n=M-1$ the summand for $l=-1$ in (\ref{u0}) equals $0.$
\begin{dfn}\label{polbas}
Define a basis in the polynomial space:
\be\label{sbas}
\phi _n (x)= \begin{cases}
  x^n  \quad {for} \quad n<M-1  \\
   \sum_{l=n-M}^n(l+n-M+2) a_{l+M-n}x^{l}  \quad {for} \quad n\geq M-1.
  \end{cases} 
   \ee
   
\end{dfn}

\begin{rem}
We chose the basis (\ref{sbas}) so that the computation of the integral  $\bigintssss \frac{\phi _n (x) dx}{\sqrt{Q(x)}} $ is obvious.

\end{rem} 
From (\ref{u0}) one gets
\[ \int \frac{\phi _n (x) dx}{\sqrt{Q(x)}}= I_n(x) +C  \quad {for} \quad n<M-1  \]
and
\beg  \int \frac{\phi _n (x) dx}{\sqrt{Q(x)}}=2x^{n+1-M} \sqrt{Q(x)}  +C  \quad {for} \quad  n\geq M-1. \eeg
The transition matrix from the standard basis
 $e_n(x)=x^n$ to the new basis  $\phi _n(x)$ is an invertible upper triangular $\infty \times \infty$ matrix. Call it  $A$. Moreover, the left top block $(M-1)\times (M-1)$ of $A$ is the unit matrix. and the left lower block 
of type $\infty \times (M-1)$ is the zero matrix. 
More precisely, we have the following:
\begin{lem}
 The entries of $A$ are given by the following formula:
\beg
A_{l,n}= \begin{cases}
\delta_{l,n} \quad for \quad n=0, \ldots M-2, \\
 0 \quad for \quad l>n \quad or \quad l< n-M .\\
 (l+n-M+2)a_{l+M-n}
\quad for \quad  n-M\leq l \leq n \quad when  \quad n\geq M \\
\qquad\qquad \qquad\qquad\qquad\quad and \quad for \quad   0\leq l \leq M-1 \quad when \quad  n=M-1.
\end{cases}
\eeg
\end{lem}
\qed

\begin{rem}
In fact, the transition matrix is  an upper triangular  matrix with the following block structure:

\beg A =
	\left[
	\begin{array}{c|c}
	I  & C  \\
		\hline
	0 & D \\
	\end{array}
	\right] ,
\eeg
where
{\scriptsize
\beg C =
\left[
	\begin{array}{cccccccc}
	a_1        &    2a_0   &         &        &         &        &       &  \\
	2a_2       &    3a_1   &  4a_0   &        &         &        &       &  \\
  3a_3       &    4a_2   &  5a_1   & 6a_0    &       &         &      &        \\
	
	
  \quad \quad \vdots \quad \quad \ddots  & \quad \quad  \ddots   & \quad \quad \ddots  & \quad \quad \ddots  &   &       &      &        \\
	
	
  \quad \quad \vdots \quad \quad \ddots  & \quad  \ddots   & \quad \ddots  & \quad \ddots  & \quad \ddots  &       &      &        \\
		
(M-1)a_{M-1} & Ma_{M-2}  & (M+1)a_{M-3} & \ldots & &  (2M-2)a_0 &    &

\end{array}
	\right] ,
\eeg
{\normalsize and}

\beg 
D=
\left[
	\begin{array}{cccccccccc}
Ma_M   &  (M+1)a_{M-1}	&  \ldots        & (2M-1)a_1 & 2M a_0 & & & & & \\
     &  (M+2) a_M      &  (M+3)a_{M-1}   &  \ldots   & (2M+1)a_1&   &(2M+2)a_0 & & &  \\
		
	&	\quad \quad  \quad \quad \ddots  & \quad \quad  \ddots   & \quad \quad \ddots   & \ldots   & \ddots  &  \ddots     &      &        \\
	
	
  & & \ddots \quad \quad \quad &  \ddots  \quad \quad \quad \quad  & \ddots  &  \ldots  & \ddots  &    \ddots   &      &        \\

\end{array}
	\right] , \eeg
	}
\end{rem}

Hence
$ \phi _n (x)=\sum_l A_{l,n}x^l   . $
Let $B$ be the inverse of $A,$ i.e.
$x^n=\sum_l B_{l,n}\phi _l(x) . $
Thus
\be\label{komb11}  \int \frac{x^n dx}{\sqrt{Q(x)}}=\sum_{l=0}^{M-2}B_{l,n}I_l+2\sum_{l=M-1}^n B_{l,n} x^{l+1-M}\sqrt{Q(x)}+C .\ee
The coefficients $B_{i,j}$ as the entries of  the inverse $B$ of the upper triangular matrix $A$ can easily be  found taking into account 
that $B$ has the following block form:

	\be\label{formB}
	B =
	\left[
	\begin{array}{c|c}
	I  & -C D^{-1}  \\
		\hline
	0 & D ^{-1}\\
	\end{array}
	\right] .
	\ee
Observe that matrices $C,D$ have  band structures. Entries in any non-zero diagonal are multiples of suitable coefficients of the polynomial $Q$ and they form an arithmetic progression with the difference being  a double of a coefficient 
of the polynomial.


\section{Hyper-elliptic integrals for a rational factor.}

In a similar way, we determine the integrals of the  family:
\beg  \int (x-p)^n\frac{ dx}{\sqrt{Q(x)}}, \quad n<-1, \;  Q(p)\neq 0 . \eeg
For a chosen $p,$ consider the linear space $V$ generated by the  family of monomials:
\beg f_n (x)= (x-p)^n  \quad {for} \quad -\infty<n<M-1 . \eeg
Taking into account  (\ref{uc1})-(\ref{uc3}) we define the following   basis of $V$:
\begin{dfn}
\beg \psi_n (x)=\begin{cases}  (x-p)^n  \quad {for} \quad -1\leq n<M-1,\\
  \sum_{l=n}^{n+M}(l+n+2) b_{l-n}(x-p)^{l}  \quad {for} \quad n< -1. \end{cases} \eeg
\end{dfn}
Hence, from (\ref{uc2}) one gets
\beg \int \frac{\psi _n (x) dx}{\sqrt{Q(x)}}= I_{n,p}(x) +C  \quad {for} \quad  -1\leq n<M-1   \eeg\label{bb1}
and
\beg  \int \frac{\psi _n (x) dx}{\sqrt{Q(x)}}=2(x-p)^{{n+1}} \sqrt{Q(x)}  +C  \quad {for} \quad  \quad n< -1. \eeg
We have the following:
\begin{lem}
Enumerate rows and columns of the transition  matrix $T$ from  $\{f_n(x)\}$ to  $\{\psi_n (x)\}$ by integral indices  $n\leq M-2$ in the decreasing order.
The matrix  $T$ is an invertible upper triangular $\infty \times \infty$ matrix .
 The left top block $M\times M$ of $T$ is the unit matrix. 
More precisely, the  matrix elements are as follows:
\beg\label{anb}
T_{l,n}= \begin{cases}\delta_{l,n} \quad for \quad n= M-2, \ldots ,0, -1, \\
 0 \quad for \quad l<n \quad and \quad l> n+M ,\\
  (l+n+2) b_{l-n} \quad for \quad n\leq l\leq n+M,  \quad n<-1  .
\end{cases}
\eeg
\end{lem}
\qed
\begin{rem}
The transition matrix is again upper triangular and has a block structure.

\beg T =
	\left[
	\begin{array}{c|c}
	I  & Y  \\
		\hline
	0 & W \\
	\end{array}
	\right] ,
\eeg
where
{\scriptsize
\beg Y =
\left[
	\begin{array}{cccccccc}
	(M-2)b_M        &               &         &        &         &        &       &  \\
	(M-3)b_{M-1}    &    (M-4)b_M   &         &        &         &        &       &  \\
  (M-4)b_{M-2}    &   (M-5)b_{M-1}   &  (M-6)b_{M}     &     &       &         &      &        \\
	
	
  \quad \quad \vdots \quad \quad \ddots  & \quad \quad  \ddots   & \quad \quad \ddots  & \quad \quad \ddots  &   &       &      &        \\
	
	\vspace{0.5cm} \\
	
		
0b_2 &-b_3  & -2b_4& \ldots & - (M-2)b_M&   & \quad   & \\

-b_1 &-2b_2  & -3b_3& \ldots & &  -Mb_M & \quad   &

\end{array}
	\right] ,
\eeg

{\normalsize and}

\beg
W=
\left[
	\begin{array}{cccccccccc}
-2b_0   &  -3b_1	&  \ldots        & -(M+2)b_M &  & &  & & \\
     &  -4b_0      & -5b_1   &  \ldots   & -(M+4)b_M&    & & &  \\
		
	&	\quad \quad  \quad \quad \ddots  & \quad \quad  \ddots   & \quad \quad \ddots   & \ldots   & \ddots \quad      &      &        \\
	
	

\end{array}
	\right] , \eeg
	}
\end{rem}

So,
$ \psi _n (x)=\sum_l T_{l,n}(x-p)^l  .$
Let $U$ be the inverse of $T,$ i,e.
$(x-p)^n=\sum_l U_{l,n}\psi _l(x) . $
Hence,
\be\label{komb22}  \int (x-p)^n \frac{ dx}{\sqrt{Q(x)}}=\sum_{l=-1}^{M-2}U_{l,n}I_{l,p}+2\sum_{l=n}^{-2} U_{l,n} (x-p)^{{l+1}}\sqrt{Q(x)}+C ,\ee

where $U$ has the following  form:
	
	\be
	U =
	\left[
	\begin{array}{c|c}
	I  & -Y W^{-1}  \\
		\hline
	0 & W ^{-1}\\
	\end{array}
	\right] .
	\ee
Similarly as the matrices  $C,D$,  the matrices $Y$ and $W$ have band structures and  entries in any diagonal  are multiples of  suitable coefficients of the polynomial $Q$ (see  (\ref{uc1}) ). We also observe, the analogous to those appearing in the matrices $C$ and $D,$ arithmetic progressions on the non-zero diagonals.
For expressing  (\ref{komb22}) as a combination of basic integrals one  replaces, for a positive $l,$ the integral $I_{l,p}$ by (\ref{newton}).}


\section{$D_N$ - action} 
\begin{dfn}
We call a  sequence of real numbers 
\begin{equation} \label{cyklmonot} (y_1, \dots y_N) \end{equation}
cyclically monotonous if there exists a cyclic permutation ${\sigma}\in {\cal S}_N$ of length $N$ such that 
\begin{equation} \label{monot} (y_{{\sigma(1)}},\dots, y_{{\sigma}_N}) \end{equation}
is strictly monotonous in the usual sense. In case where (\ref{monot}) is strictly increasing we say that   (\ref{cyklmonot})  is cyclically increasing and analogously we define a cyclically decreasing sequence.
\end{dfn}
\begin{rem}
On a real projective line there exists a canonical positive orientation on $S^1\cong P^1_R$ coming from  the positive direction of the real line ${\mathbb R}$.  Notice that a cyclically monotonous sequence of real numbers yields an orientation on $S^1\cong P^1_R$. In case of a cyclically increasing sequence  this orientation is positive (cf. Fig. 1).  Cyclically decreasing sequence leads to negative orientation.\end{rem}

Consider  the  set of permutations
$\{\tau_k, \eta_k , k=1,\ldots , N \}\subset {\cal S}_{N}$  defined by the action on the $N$-tuples of real numbers, viewed as the elements of the real projective space ${\mathbb P}^1_{\mathbb R}\cong S^1$:
\beg \tau_k (y_1,\ldots , y_N)=(y_{k+1}, y_{k+2}, \ldots , y_N, y_1, \ldots, y_k), \eeg
\beg \eta_k (y_1,\ldots , y_N)=(y_{k}, y_{k-1}, \ldots , y_1, y_N, \ldots, y_{k+1}). \eeg
The permutations 
 $\tau_k, \eta_k$   transform cyclically monotonous sequences into cyclically monotonous sequences. In fact ${\tau}_k$ transform cyclically decreasing (resp. increasing) sequences into cyclically decreasing (resp. increasing) sequences.   Permutations  $\eta _k$  reverse cyclic monotonicity.
This set forms a subgroup of ${\cal S}_N$ isomorphic to the dihedral group $D_{N}.$ 

\begin{rem}\label{rem51} Notice that
that   ${\tau}_k$ preserves and ${\eta}_k$ reverses an orientation of  $S^1$ derived from a cyclically monotonous sequence.
\end{rem}

Let $(a_1,\ldots , a_N)$ be an  increasing sequence of  roots of a polynomial   $Q.$ 
We consider integrals of the form
  (\ref{wystarczy}) and (\ref{enough}). In the next section we show how to transform, for a cyclically monotonous sequence  of roots of $Q(x),$ hyper-elliptic integrals  (\ref{wystarczy}) and (\ref{enough}) into the Riemann canonical form. The choice of a transformation into the canonical form depends on the interval in which a variable of integration $x$  is supposed to be. Assume $x\in (a_N, \infty)$ or $x\in (-\infty, a_1)$ then to assure
that  
  \be \label{przedzial} x\in (a_k, a_{k+1}) \ee
we choose  the transformation described in Corollary \ref{corR}
 for either
\be \label{przes} (x_1, \ldots , x_N)=\tau_k (a_1,\ldots , a_N)\ee
or
\be \label{odbicie} (x_1, \ldots , x_N)=\eta_k (a_1,\ldots , a_N).\ee
With the change of interval to which $x$ belongs the index $k$ changes in 
 (\ref{przedzial}) (cf. Fig.1).

\begin{figure}
\setlength{\unitlength}{0.1in}
  \begin{picture}(20,15) 
\put(10,7){\circle{10}}
\put(5,7){\circle*{0.3}}
\put(2,7){$x_{k+2}$}
\put(4.8,11.5){$x_{k+1}$}
\put(15,7){\circle*{0.3}}
\put(6.9,11){\circle*{0.3}}
\put(7,3){\circle*{0.3}}
\put(5,2){$x_N$}
\put(11,2.1){\circle*{0.3}}
\put(10.8,0,8){$x$}
\put(5.8,-1){$L(x_N,x_1)$}
\put(12,2.45){\circle*{0.3}}
\put(12,1){$x_1$}
\put(14.0,3.9){\circle*{0.3}}
\put(14.5,2.6){$x_2$}
\put(13.4,0.2){$L(x_1,x_2)$}
\put(16,7){$x_{k}$}

\put(3.5, 5.5){\circle*{0.15}}
\put(4.15, 3.5){\circle*{0.15}}
\put(3.7, 4.5){\circle*{0.15}}

\put(15.4, 3.8){\circle*{0.15}}
\put(15.8, 4.7){\circle*{0.15}}
\put(16.0, 5.7){\circle*{0.15}}

\put(12.5,11){\line(1,1){0.4}}
\put(13.3,11.5){$[\pm{\infty}]$}
\put(18,11){$L(x_k,x_{k+1})$}
\put(-3,10){$L(x_k,x_{k+1})$}

\put(5.5,9.3){\line(0,1){0.4}}
\put(5.5,9.3){\line(1,0){0.4}}

\end{picture}

\caption{Monotonous sequence in $P^1_R$}
\label{fig:1}       
\end{figure}
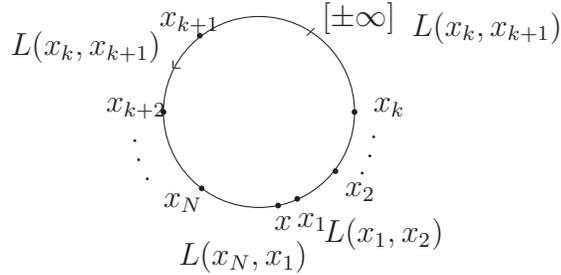

The roots  $(x_1,x_2,,\ldots , x_N))$ divide the projective line ${\mathbb P}_{\mathbb R} ^1$ into $n$ arcs $L(x_1,x_2)$, $L(x_2,x_3)$, \ldots, $L(x_{N-1},N)$, $L(x_N, x_1)$,  positively or negatively oriented  depending  on the type of cyclic monotonicity  of the sequence  $(x_1,x_2,\ldots ,x_N)$.    
let $x$ be a number different from the roots. Among all possible cyclically monotonous sequences of the roots we choose as $x$-canonical those which satisfy
$x\in L(x_N, x_1)$.  There are only two  such sequences: one is cyclically increasing, the other is cyclically decreasing. 
Moreover, each of the $x$-canonical sequences can be obtained from any cyclically monotonous sequence of roots by applying  a suitable transformation:  either ${\tau}_k$ or ${\eta}_k.$


Accordingly, there is a change  in enumeration of the roots due to the condition
 (\ref{przedzial}).
 We also see that  in the orbit of the action of the group    $D_N,$ all possibilities of 
 (\ref{przedzial}) are obtained  twice,  cf. (\ref{przes}) and (\ref{odbicie}). 

Recall that the cross-ratio of the numbers $d_1, d_2, d_3, d_4$ is given by the following formula:
 \beg (d_1, d_2; d_3, d_4)=\frac{(d_3-d_1)(d_4-d_2)}{(d_3-d_2)(d_4-d_1)}.   \eeg

\begin{lem}\label{lemat51} Let   $(a,b,c)$ be a cyclically increasing (resp. decreasing). Then the following homographic transformation: 
\[ f:  P^1_R \longrightarrow  P^1_R, \quad f(x)=(b,c;a,x) \]
preserves (resp. reverses) an orientation of the real projective line.
\end{lem}
\begin{proof}.
It is enough to notice that for  $(a,b,c)$ a cyclically increasing (resp. decreasing) sequence the derivative $f'(x)$ is positive (resp. negative) for $x\neq b$. 
Hence  the homography $f$ as a map $R\longrightarrow R$  is locally strictly increasing (res. decreasing)  depending on the type of cyclic monotonicity  of the sequence $(a,b,c)$. 
Gluing the points $\pm \infty \,$, i.e. $ [\infty]=\{ \pm \infty \}$ , (cf. Fig.1) in the domain and range we obtain the assertion.
\end{proof}

\begin{lem}\label{lemat52}
Let  ${\mathbf x}=(x_1, \ldots , x_N)$, $N\geq 4$ be a cyclically monotonous sequence. Then the following sequence:
\[ t_k:= (x_{N-1}, x_N; x_1,x_k)\; \quad k=2,\ldots , N-2 \]
is strictly increasing with terms greater than $1$.
\end{lem}
\begin{proof}
The lemma follows from Lemma  \ref{lemat51}, $\quad $ Remark  \ref{rem51}
  and equalities \\
	$(x_{N-1}, x_N; x_1,x_1) =1,  (x_{N-1}, x_N; x_1,x_{N-1})=[\infty].$
\end{proof}

\section{Riemann canonical form for hyper-elliptic integrals.}

Let 
\beg A=\left[ \begin{array}{cc} a & b \\ c & d \end{array} \right] , \quad \psi _A (t)=\frac{at+b}{ct+d} .\eeg
Denote $N_A(t)=at+b$ and $D_A(t)=ct+d.$

\begin{dfn}\label{defa}
Define an operator $r_k(A)$ sending  a real function $f$ of one real variable $x$  to  a real function of one real variable $t$
by the following  formula
\be\label{rjeden}
(r_k(A)f)(t)=D_A ^k(t) f(\psi _A (t)).
\ee
\end{dfn}

The following proposition  describes basic properties of this operator 
\begin{prop}\label{op} Operator $r_k(A)$ enjoys the following properties
\begin{enumerate}
\item[1)]
 For any numbers $k,l$ and functions $f,g$ the following equality holds:
$$ \left(r_k(A)f \right) \left( r_l(A)g\right)=r_{k+l}(fg),$$
\item[1a)]
in particular taking $f\equiv 1$ in 1)  we get:
$$  (r_{k+l}(A) g)(t)=D_A ^k(t)(r_{l}(A) g)(t)$$
\item[2)]
 Operator $r_1(A)$ sends  a linear polynomial $P_1(x)=x-x_0$  to the following  linear in $t$ polynomial:
$$
\left( r_1(A)P_1\right)(t)=N(t)-D(t)x_0,
$$
\item[3)]
$$r_m(A){{\prod}_{i=0}^{m-1}(x-x_i)}={{\prod}_{i=0}^{m-1}}r_1(A)(x-x_i)={\prod}_{i=0}^{m-1}(N(t)-D(t)x_i),$$
\item[4)]
 If $f(x)=\sum _j c_jx^j$ is a polynomial of degree not bigger than $k$, the function $r_{k}(A) f (t)$ is also a polynomial and

$$ (r_{k}(A) f) (t)=\sum _j c_j N_A^j (t) D_A ^{k-j}(t) .$$
\end{enumerate}
\end{prop}
\begin{proof}
1) follows from definition \ref{defa}, 2) is a straightforward calculation, 3) follows from 2) and 1). Since $r_{k}(A)$ is by definition additive  4) follows from 1') and 3)  applied to $x^j, j=0,\dots , k.$
\end{proof}

Consider the following differential form: 
\be \label{omga} \omega _x= \frac{R(x) dx}{|\sqrt{P(x)}|} ,\ee
where $P(x)=a_N{\prod}_{i=1}^{N}(x-x_i)$, $N=2m$.
\begin{lem}\label{pb}
 The pullback of the form ${\omega}_x$ by the homographic map ${\psi}_A$ has the following form:
\beg \label{pullback1} ({\psi _A }^*\omega )_t=  \det A\,  ({\psi _A }^*R)(t) \frac{|D_{A}^{m-2}(t)|dt}{\sqrt{|(r_{2m}(A)P)(t)}|}.\eeg
\end{lem}

\begin{proof}
Notice that by (\ref{rjeden}) we have $P(A(t))=D_A^{-2m}(r_{2m}(A)P)(t).$ The formula follows now by substitution.
\end{proof}

\begin{rem}
Notice  that for the case $P(\frac{a}{c})=0$, ${r_{2m}(A)P(t)}$ is a polynomial of a degree  $N-1$ where $N=2m$ is a degree of $P$.
\end{rem}
Now let $A$ be such that 
$\psi _A (\infty)=x_{N-1}, \psi _A (0)=x_N, \psi _A (1)=x_1.$  Of course $A$ is defined up to a scalar factor.
We can take  as $A$ the matrix given by the following equality:
\be  A=A_{x_{N-1},x_N,x_1}=\left[ \begin{array}{cc}  a&  b \\ c & d \end{array} \right] ,\ee
where
\be \label{macA} a=(x_1-x_N)x_{N-1}, b=-(x_1-x_{N-1})x_N, c=x_1-x_N, d=-(x_1-x_{N-1}).\ee

We have the following:
\begin{thm}\label{sub} Let  $P(x)=a_{N}{\prod}_{i=1}^{2m}(x-x_i),$ where $(x_1,\dots ,x_{2m})$ is a cyclically monotonous sequence, be a polynomial of degree $N=2m$ and let $A$ be as in (\ref{macA}).
Then
\beg (r_N(A)P) (t)=C t(1-t)(1-k_2t)\cdot \ldots \cdot (1-k_{N-2}t) , \quad 1>k_2>\ldots >k_{N-2}, \eeg
where 
the variable  $t$ is expressed by $x$ as the cross-ratio
\beg t=(x_{N-1}, x_N;x_1, x)  \eeg
and 
\[k_j=t_j^{-1},  \]
where
\be t_j=(x_{N-1}, x_N;x_1, x_j), \quad j=2, \ldots , N-2 \ee
are  roots of $ r_N(A)P$ different from $0$ and $1$. $C$ is here  a suitable constant cf.  (\ref{stala}).
\end{thm}
\begin{proof}
Indeed,  


\beg P_A(t)=a_N[(at+b)-(ct+d)x_1]\cdot \ldots \cdot [(at+b)-(ct+d)x_N] ,\eeg
where $a,b,c,d$ are given by (\ref{macA}).
Hence 
\[ (at+b)-(ct+d)x_i=(x_1-x_N)(x_{N-1}-x_i)t-(x_1-x_{N-1})(x_N-x_i) \]
and $k_j=t_j^{-1}$ for $j=2,\dots ,2N-2.$
By Lemma \ref{lemat52} we obtain $1>k_2>\ldots >k_{N-2}.$

Also
\be\label{stala} C=a_N(x_N-x_1)(x_N-x_{N-1}) \prod_{j=1}^{N-1} (x_{N-1}-x_1)(x_N-x_j) .\ee

\end{proof}

\begin{lem}\label{corR}
Let $\omega _x$ be as in (\ref{omga}) and $A$ - as in (\ref{macA}). Then 
\be \label{kanon} {\psi _A }^*\omega _t=  \frac{ ({\psi _A }^*R)(t) |D^{m-2}(t)| \epsilon}{\sqrt{|a_N(x_{N-1}-x_1)^{N-3}\prod_{j=2}^{N-2}(x_N-x_j)}|}  \,\,\times \ee
\[   \frac{dt}{\sqrt{t(1-t)(1-k_2t)\cdot \ldots \cdot (1-k_{N-2}t)}},\]
where 
$ \label{znak} \epsilon =\sgn \det A . $
\end{lem}
\begin{proof}
We have the following equality:
\[ \det A=(x_N-x_1)(x_N-x_{N-1})(x_{N-1}-x_1) .\]
Now use Theorem \ref{sub} and Lemma \ref{pb}. 
\end{proof}

The last lemma one can read as such a formula:
\be   \label{canonint} \int\frac{R(x) dx}{\sqrt{P(x)}}=
\frac{ \epsilon  }{\sqrt{|a_N(x_{N-1}-x_1)^{N-3}\prod_{j=2}^{N-2}(x_N-x_j)}|}  \,\,\times \ee
\[  \int \frac{({\psi _A }^*R)(t) |D^{m-2}(t)|dt}{\sqrt{t(1-t)(1-k_2t)\cdot \ldots \cdot (1-k_{N-2}t)}} .\]

{\bf Remark.}
{\it Observe that $\epsilon=1$  (resp. $\epsilon=-1$) for $(x_1, \ldots  ,x_{N-1}, x_N)$ cyclically increasing (resp. cyclically decreasing).}

We finish this section with the following, easy to prove lemma:

\begin{lem}
Let  $P(t)=t(1-t)(1-k_2t)\dots (1-k_{n-2}t)= {\sum}_{i=1}^n a_it^i.$ Then 
\beg
a_i=(-1)^{i-1}{\sigma}_{i-1}(1,k_2,\dots ,k_{n-2}),
\eeg
where $$ {\sigma}_i(u_1,\dots , u_n)={\sum}_{1\leq s_1<s_2<\dots <s_i \leq n} u_{s_1}\dots u_{s_i}$$
is an elementary symmetric polynomial of degree $i$, for $i>0$ and ${\sigma}_0(u_1,\dots, u_n):=1.$
\end{lem}
\qed


\section{Deriving some formulas}

\subsection{Elliptic case}
In the case when the degree of the polynomial $Q_4:=P$ is  $4$: 
\[ Q_4=a_4\prod_{i=1}^4 (x-x_i) ,\]
the formula  (\ref{canonint}) can be simplified
We assume that the sequence of roots $(x_1,x_2,x_4,x_5)$ is cyclically monotonous. 
\beg \label{kanon4} \int \frac{R(x) dx}{\sqrt{|Q_4(x)}|}= \frac{1}{\sqrt{|a_4(x_{3}-x_1)(x_{4}-x_2)|}} 
 \int \frac{ R\left(\frac{at+b}{cx+d} \right) dt}{\sqrt{|t(1-t)(1-kt)|}}. \eeg
Here the substitution $x=\phi  (t)=\frac{at+b}{cx+d}$ is defined by $\phi (1)=x_1, \phi ([\infty])=x_3, \phi (0)=x_4$, 
hence coefficients $a$, $b$, $c$, $d$ can be given by (\ref{macA}) for $N=4$.
 It is easy to see that the quantity under the square root in the denominator in the expression in front of the integral
$ |a_4(x_{3}-x_1)(x_{4}-x_2)|= |a_4(a_{3}-a_1)(a_{4}-a_2)|$ 
 is    invariant  under both  $\tau_k$ and $\eta_k$.

Let us consider the following three particular cases:
\[ R(x)\equiv 1, \quad R(x)=x, \quad R(x)=\frac{1}{x-p}. \]

 In the sequel,  we  use properties concerning homographic transformations, described in the Appendix.

For the exponent  $-1$ we will use the following integral:
\beg P(t,h,k)=\int\frac{dt}{(1-ht)\sqrt{t(1-t)(1-kt)}}, \quad   P(t,h,k)=-\frac{1}{h}I_{-1,\frac{1}{h}}(t,k)\eeg
where
\beg I_{-1,s}(t,k)=\int\frac{dt}{(t-s)\sqrt{t(1-t)(1-kt)}}.\eeg


Using formulas from the Appendix we obtain
\be\label{pirsza} \int \frac{ dx}{\sqrt{|Q_4(x)}|}= \frac{\epsilon}{\sqrt{|a_4(x_{3}-x_1)(x_{4}-x_2)|}} I_0(t,k) ,\ee

Further, in terms of  $P$  we obtain the following formulas:

\be \label{xdx1} \int \frac{ xdx}{\sqrt{|Q_4(x)}|}
 =\frac{\epsilon}{\sqrt{|a_4(x_{3}-x_1)(x_{4}-x_2)|}} \left(x_3 I_0(t,k)+(x_4-x_3)
P\left(t,h,k\right)\right),\ee

\be \label{last} \int \frac{ dx}{(x-p)\sqrt{|Q_4(x)}|}=\frac{\epsilon}{\sqrt{|a_4(x_{3}-x_1)(x_{4}-x_2)|}(x_3-p)(x_4-p)}  \ee
\[ \left((x_4-p) I_0(t,k)-(x_4-x_3)P\left(t,h_p,k\right)\right).\]

 \medskip

In the formulas (\ref{pirsza})-(\ref{last}) we put

\be \label{teka} t=(x_{3}, x_4;x_1, x) , \quad k=(x_{3}, x_4;x_1, x_2)^{-1},\ee
\be \label{hhp} \quad h=\frac{x_4-x_1}{x_3-x_1}, \quad h_p=(x_3,x_4;x_1,p)^{-1}. \ee

The formulas (\ref{pirsza})-(\ref{last}) can be transformed by the operations $\tau _i, \eta _i$, $i=1,2,3,4$ which preserve the relative position of roots on $S^1.$ Thus we see that the  dihedral group  $D_4$ acts freely on 
these integral formulas. More precisely, every formula of  
(\ref{pirsza})-(\ref{last}) determines a regular orbit of this group action. 

Now we illustrate the action of the group $D_4$ on the basis of  elliptic  integrals we have chosen.

\subsection{ Definite elliptic integrals.}

  Definite integrals of the forms $\frac{dx}{\sqrt{G_{4}(x)}}$, $\frac{xdx}{\sqrt{G_{4}(x)}}$ and  $\frac{dx}{(x-p)\sqrt{G_{4}(x})}$ can be written by means of indefinite ones, which were calculated in the previous subsection.

Recall that elliptic integrals of first and third kinds are defined by the following formulas:


\be \label{fk}
F(\phi,l)={\int}_0^{\phi}\frac{d\alpha}{\sqrt{1-l^2{\sin}^2\alpha}}
=\frac{1}{2}{\int}_0^{{\,\sin^2}{\phi}}\frac{dt}{\sqrt{t(1-t)(1-l^2t)}}.\ee

\be \label{tk}
\Pi (\phi, h, l)={\int}_0^{\phi}\frac{d\alpha}{(1-h\sin^2 \alpha)\sqrt{1-l^2{\sin}^2\alpha}} \ee
\[=\frac{1}{2}{\int}_0^{{\,\sin^2}{\phi}}\frac{dt}{(1-ht)\sqrt{t(1-t)(1-l^2t)}}.\]

The roots  $(x_1,x_2,x_3,x_4)$ divide the projective line ${\mathbb P}_{\mathbb R} ^1$ into four arcs $L(x_1,x_2)$, $L(x_2,x_3)$, $L(x_3,x_4)$ and $L(x_4,x_1)$ (Fig.1. for $N=4$),  positively or negatively oriented  depending  on the type of cyclic monotonicity  of the sequence  $(x_1,x_2,x_3,x_4)$. For any  $u\in L(x_4,x_1)$   we also distinguish an oriented sub-arc $L(x_4,u)$. 
Naturally,
$ \int_{L(x_4,u)} =\int_{x_4}^u$ if  $[\infty] \notin L(x_4,u)$. In the case where  $[\infty] \in L(x_4,u)$ we have   $\int_{L(x_4,u)} =\int_{x_4}^\infty +\int_{-\infty}^u$  if the arc has positive orientation,  and $\int_{L(x_4,u)} =\int_{x_4}^{-\infty} +\int_{\infty}^u$ if it is negatively oriented (cf. Fig.1).

In our notation the  equations  (\ref{pirsza} )-(\ref{last}) lead to:
\begin{equation}\label{firstkind1}
{\int}_{L(x_4,u)} \frac{dx}{\sqrt{Q_4(x)}}=\frac{2\epsilon}{\sqrt{|a_4(x_4-x_2)(x_3-x_1)|}}F(\nu, q),
\end{equation}

\be \label{xdxdef} \int_{L(x_4,u)} \frac{ xdx}{\sqrt{|Q_4(x)}|}
 =\frac{2\epsilon}{\sqrt{|a_4(x_{3}-x_1)(x_{4}-x_2)|}} \left(x_3 F(\nu, q)+(x_4-x_3) \Pi(\nu,h, q) \right),\ee

\be \label{lastdef} \int_{L(x_4,u)} \frac{ dx}{(x-p)\sqrt{|Q_4(x)}|}=\frac{2\epsilon}{\sqrt{|a_4(x_{3}-x_1)(x_{4}-x_2)|}(x_3-p)(x_4-p)}  \ee
\[ \left((x_4-p) F(\nu, q)-(x_4-x_3)\Pi(\nu,h, q)\right),\]
where
\begin{equation}\label{A9}
\nu= {\arcsin} \sqrt{(x_3, x_4; x_1, u)}, \qquad   \qquad q= \sqrt{(x_3, x_4 ; x_1, x_2)^{-1}}
\end{equation}
(cf. \ref{teka}). Both $h$ and $ h_p$ are given by (\ref{hhp}).

Consider the formulas (1)-(8)  from the sections 3.147, 3.148, 3.151  of  \cite{gr00}. 
We obtain the cases (8) in the formulas
 3.147, 3.148, 3.151 by taking  $x_1=d$, $x_2=c$, $x_3=b$, $x_4=a$ in  (\ref{firstkind1})-(\ref{lastdef}). 
One readily verifies that applying  ${\tau}_i $ and ${\eta}_i$ for $ i=1,\dots 4$ to (\ref{firstkind1})-(\ref{lastdef}) one obtains all of the formulas  of 3.147, 3.148 and 3.151.
( the orbit of any integral  under the of $D_4$-action  yield all the formulas in the corresponding  section).

In the formulas  discussed above we have not  used recursive formulas.  The expression of an elliptic integral  as a combination of basic integrals was connected with  both the properties  of homographic transformations, described in the Appendix, and the change of  variable.

Notice that analogously,  the action of $D_{4}$ 
can be used for general,  much more complicated, elliptic integrals i.e. those which require recurrence.
More generally, in order
 to obtain fewer formulas, one can use the described above  $D_N$-action  for a hyper-elliptic  case.

We do not focus on a recurrence for the elliptic case.  Our approach in this  case is more efficient than the usual recursive procedures that can be found in the literature cf. 
 \cite{bf71}, \cite{ps97}.  In the next subsection we give two examples concerning hyper-elliptic integrals which  show  advantages of our approach.

\subsection{Examples of computation of hyper-elliptic integrals.}

We end this section with two examples. 
In Section 6 we have showed how in general hyper-elliptic case one transforms the integral involving the polynomial $Q$ to the Riemann canonical form. Therefore we will not do it here,
but we show how to express an integral as a  linear combination of basic integrals.
 In the first  example, we compute the  hyper-elliptic integral with a polynomial factor of  higher degree. In the second example we consider 
a hyper-elliptic integral with a rational factor in the Riemann form. 
In the first  example, we compute the  hyper-elliptic integral with a polynomial factor of  higher degree, In the second example we consider 
a hyper-elliptic integral with a rational factor in the Riemann form. 
We included both examples to illustrate how efficient and easy to apply is our approach cf. Remark \ref{lastr}.

\begin{example}\label{ex73}
Let us compute the integral
 \be\label{Rform} \int \frac{x^9 dx}{\sqrt{Q(x)}}, \qquad {\mathrm{where}} \qquad
 Q(x)=\sum_{j=0}^7 a_j x^j. \ee
  Notice that under the assumption that all roots of a polynomial $Q(x)$ are real and distinct one can use a linear transformation for  obtaining the Riemann canonical form of  (\ref{Rform}).
 We leave the justification of this to the reader.
 Thus we have $n=9$ and $M=7.$ However, in the formulas below i.e., (\ref{w79})-(\ref{p1})  we keep writing $M$ to make clear how they were derived. 
According to formula \eqref{komb11}
 we obtain
\be\label{p79}
 \int \frac{x^9 dx}{\sqrt{Q(x)}}=\sum_{l=0}^{5}B_{l,9}I_l+2\sum_{l=6}^9 B_{l,9} x^{l-6}\sqrt{Q(x)}+C. \ee

We have to find the elements of the $9$-th column of the matrix $B.$ We start with the diagonal term and move inductively up. 
According to  \eqref{formB} the elements $B_{9,9}, \ldots , B_{6,9}$ are the terms of the matrix $D^{-1}$, whereas the remaining terms $B_{l,9}, \; l=0, \ldots 5$ are the elements of the matrix $-CD^{-1}$.
We have
\be \label{w79} B_{9,9}=\frac{1}{A_{9,9}}=\frac{1}{D_{9,9}}, \quad B_{9,9}=\frac{1}{(M+6)a_M} ,\ee
and further recursively
$ B_{k,9}=\frac{-1}{D_{k,k}}\sum_{l=k+1}^9 D_{k,l}B_{l,9} ,$ for $ 5<k<9. $
Thus we obtain the following formulas:
\be \label{p1} B_{8,9}=-\frac{(M+5) a_{M-1}}{(M+4)(M+6){a_M}^2} ,\ee
\beg B_{7,9}= \frac{1}{(M+2)(M+4)(M+6){a_M}^3} \left[(M+3)(M+5){a_{M-1}}^2-(M+4)^2a_{M-2}a_M \right] \eeg
\beg \label{k1} B_{6,9}= \frac{-1}{M(M+2)(M+4)(M+6){a_M}^4}\left\{(M+1)(M+3)(M+5) {a_{M-1}}^3  -  \right. \eeg
\[ - \left[ (M+1)(M+4)^2+(M+2)^2(M+5)\right]a_{M-2}a_{M-1}a_M+ \]
\[ \left. + (M+2)(M+3) (M+4) a_{M-3}{a_M}^2 \right\} . \]
and
{\scriptsize 
\be\label{p2} \left[\begin{array}{c}  B_{0,9} \\ B_{1,9} \\ B_{2,9} \\ B_{3,9}\\ B_{4,9} \\B_{5,9}  \end{array} \right]= -
\left[\begin{array}{cccc}  
 a_1 & 2 a_0 & 0 & 0 \\
2a_2 & 3a_1  & 4a_0 & 0 \\
3a_3 & 4a_2 & 5a_1 & 6a_0 \\
4a_4 & 5a_3 & 6a_2 & 7a_1 \\
5a_5 & 6a_4 & 7a_3 & 8a_2 \\
6a_6 & 7a_5 & 8a_4 & 9a_3
\end{array} \right] 
\left[\begin{array}{c}  B_{6,9} \\ B_{7,9} \\ B_{8,9} \\ B_{9,9}  \end{array} \right].
\ee 
}
Inserting (\ref{w79}),(\ref{p1}) and (\ref{p2}) to (\ref{p79}) we obtain an explicit form of the desired integral.
\end{example}
\begin{ex}\label{ex74}
Consider the integral
\beg
\int\frac{(t-\frac{3}{2})^{-3}\,\,dt}{\sqrt{t(1-t)(1-\frac{1}{4}t)(1-\frac{1}{3}t)(1-\frac{1}{2}t)}}.
\eeg
Then 
$
P(t)={\sum}_{i=0}^5a_it^ i= {\sum}_{i=0}^5b_i(t-\frac{3}{2})^i,
$
where
$a_0=0,\, a_1=1,\, a_2=-\frac{25}{12},\, a_3=\frac{35}{24}, \,a_4=-\frac{5}{12},\, a_5=\frac{1}{24}, $
$ b_0=-\frac{15}{256}, \, b_1=\frac{3}{128}, \, b_2=-\frac{25}{98},\, b_3=-\frac{5}{48},\, b_4=-\frac{5}{48},\, b_5=\frac{1}{24}.$
\end{ex}
 The  subspace $V_{-2}\subset V$ generated by $(x-\frac{3}{2})^n, \quad  n=3,2,,\dots ,-1,-2,-3$  yields the corresponding invariant subspace of integrals
 $I_3, I_2,,\dots I_0, I_{-1,\frac{3}{2}},  I_{-2,\frac{3}{2}}, I_{-3,\frac{3}{2}}.$ Therefore instead of infinite matrices $Y$ and $W$ we can take their cuts (denoted by the same letters for simplicity):
 
{\scriptsize
 \beg
  W  = \left[
	\begin{array}{ccc}
	-2b_0        &     -3b_1        \\
	\\
 0   &    -4b_0   
\end{array}\right]
	=
	 \left[
	\begin{array}{ccc}
\frac{15}{128}       &     -\frac{9}{128}        \\
\\
	0   &    \frac{15}{64}    
\end{array}\right] \qquad
  W^{-1}  = \
	 \left[
	\begin{array}{ccc}
\frac{128}{15}       &     \frac{64}{25}        \\
\\
	0   &    \frac{64}{15}    
\end{array}\right],
\eeg 
}

{\scriptsize 
\beg
 Y  = \left[
	\begin{array}{ccc}
	3b_5        &     0         \\
	\\
	2b_{4}    &    b_5   \\
	\\
     b_{3}    &   0b_4        \\
     \\
	0b_2 &-b_3   \\
	\\
-b_1 &-2b_2    
\end{array}\right]
	=
	 \left[
	\begin{array}{ccc}
\frac{3}{24}       &     0         \\
\\
	-\frac{10}{48}    &    \frac{1}{24}  \\
	\\
     -\frac{5}{48}    &   0       \\
     \\
	0 &\frac{5}{48}   \\
	\\
-\frac{3}{128} &-\frac{25}{48}    
\end{array}\right] ,\quad
 \qquad
- YW^{-1}  = 
	 \left[
	\begin{array}{ccc}
-\frac{16}{15}       &     -\frac{8}{25}         \\
\\
	\frac{16}{9}    &    \frac{16}{45}  \\
	\\
     \frac{8}{9}    &   \frac{4}{15}      \\
     \\
	0 &-\frac{4}{9}   \\
	\\
\frac{1}{5} &\frac{1027}{450}    
\end{array}\right], \eeg
}
Thus   in the following, derived from (\ref{komb22}), formula:

\medskip

\begin{multline}\label{exmpl}
\int\frac{(t-\frac{3}{2})^{-3}\,\,dt}{\sqrt{t(1-t)(1-\frac{1}{4}t)(1-\frac{1}{3}t)(1-\frac{1}{2}t)}}=\\
{\sum}_{l=-1}^{3} U_ {l,-3}\int\frac{(t-\frac{3}{2})^{l}\,\,dt}{\sqrt{t(1-t)(1-\frac{1}{4}t)(1-\frac{1}{3}t)(1-\frac{1}{2}t)}} \\
+{{2}}{\sum}_{l=-3}^{-2}U_{l,-3}\,(t-\frac{3}{2})^{{l+1}}\,\,\,{{\sqrt{t(1-t)(1-\frac{1}{4}t)(1-\frac{1}{3}t)(1-\frac{1}{2}t)}}}+C. 
\end{multline}
we should put $U_{3,-3}=  -\frac{8}{25} ,  U_{2,-3}= \frac{16}{45},    U_{1,-3}= \frac{4}{15},   U_{0,-3}= \frac{4}{9},   U_{-1,-3}= \frac{1027}{458},  U_{-2,-3}= \frac{64}{25},  U_{-3,-3}= \frac{64}{15}.$
Further, the formula (\ref{newton}) enables one to replace the integrals $I_{l,p}$ for positive indices $l$ by appropriate combinations of the basic  integrals $I_l.$

\begin{rem}\label{lastr}
Notice that  one does not need to invert whole matrix $W$. It is enough to compute an appropriate column of $W^{-1}.$ This is  important in numerical calculations involving computations of many hyper-elliptic integrals $I_{-n,p}$ for 
big $n.$ 
\end{rem}

{
 \begin{rem}
Let 
\begin{equation}\label{Lur1}
F_D^{(n)}\left(\begin{matrix}a,&b_1,\dots , b_n \\  & c  \end{matrix}\mid x_1,\dots , x_n \right) = {\sum}_{i_1,\dots , i_n = 1}^{\infty}\frac{(a)_{i_1+\dots +i_n} {{(b_{1})}_{i_1}}\dots  {({b_{n})}_{i_n}}}{(c)_{i_1+\dots +i_n}\cdot {i_1}!\cdot \dots  {i_n}!}x_1^{i_1}\dots
\, x_n^{i_n},
\end{equation}
where  $(s)_{i}$ is a Pochhammer symbol,
denote the Lauricella function od type $D$  in $n$ variables. Then  for  and   $c>a>0$
\begin{equation}\label{Lur2}
F_D^{(n)}\left(\begin{matrix}a,& b_1, \dots , b_n\\  & c \end{matrix} \mid x_1,\dots , x_n\right) = K{\int}_{0}^{1} t^{a-1}(1-t)^{c-a-1}(1-x_1t_1)^{-b_1}\dots (1-x_nt_n)^{-b_n}dt
\end{equation}
where $K=\frac{\Gamma (c)}{{\Gamma (a)} {\Gamma (c-a)}} .$
Taking the definite integral in (\ref{exmpl}) over the interval $[0,1]$ we obtain the following equality for the special values of Lauricella functions:
\begin{equation*}
F_D^{(4)}\left(\begin{matrix}\frac{1}{2}, &3,\frac{1}{2}, \frac{1}{2}, \frac{1}{2} \\& 1 \end{matrix} \mid \frac{2}{3},  \frac{1}{4},  \frac{1}{3},  \frac{1}{2}\right)= {\sum}_{l=-1}^3 \left({\frac{-3}{2}}\right)^{l+3}U_{l,-3}F_D^{(4)}\left( \begin{matrix}\frac{1}{2}, &-l,\frac{1}{2}, \frac{1}{2}, \frac{1}{2}&\\ &1 \end{matrix}\mid \frac{2}{3},  \frac{1}{4},  \frac{1}{3},  \frac{1}{2}\right).
\end{equation*}
Notice also that formulas (\ref{komb11}) and (\ref{komb22}) lead to many,  analogous to the above,  formulas for special values of Lauricella functions.
\end{rem}
}

\appendix

\section{}
In this appendix, for convenience of the reader, we collect some elementary facts concerning real homographic transformations.
Let 
\be \label{hom1}
{\phi}(t)=\frac{at+b}{ct+d}\,\,,    a,b,c,d  \in {\mathbb R}
\ee
be a homographic transformation of ${\mathbb P}^1_{\mathbb R}.$
\begin{lem}\label{hom30} 
\begin{itemize}
\item[i)] Any homographic transformation (\ref{hom1}) of ${\mathbb P}^1_{\mathbb R}$ can be written in the following form:
\beg\label{hom2}
{\phi}(t)= {\phi}(\infty) + {\frac{{\phi}(0)-{\phi}(\infty)}{1-\frac{t}{{\phi}^{-1}(\infty)}}},
\eeg
\item[ii)] the inverse homographic transformation is given by the following cross-ratio:
\be\label{hom3}
t={\phi}^{-1}(x)=({\phi}(\infty), {\phi}(0);{\phi}(1), x).
\ee

\end{itemize}
\end{lem}
\begin{proof} 
For i) write ${\phi}(t)$ in the following form
\be\label{hom4}
\frac{at+b}{ct+d}=\frac{a}{c}+\left({\frac{b}{d}-\frac{a}{c}}\right)\frac{1}{1+\frac{c}{d}t} 
\ee
and notice that ${\phi}(\infty)=\frac{a}{c},\,\, {\phi}(0)=\frac{b}{d}$ and ${\phi}^{-1}(\infty)=-{\frac{d}{c}}.$

For ii) put $x={\phi}(0)$ (resp. $x={\phi}(1)$ and $x={\phi}(\infty)$ ) in the cross-ratio (\ref{hom3}) and check that  the value is $0$ (resp. $1$ and ${\infty}$).
\end{proof}
As a consequence one obtains the following:

\begin{lem}\label{hom41}
\begin{itemize}
\item[i)] For any $p\neq \phi(0), \phi(\infty),$
\beg \frac{1}{\phi(t)-p}=\frac{1}{[\phi(0)-p][\phi(\infty)-p]}\left\{[\phi(0)-p]+\frac{\phi(\infty)-\phi(0)}{1-\frac{t}{\phi^{-1}(p)}} \right\},\eeg
\item[ii)] \beg \frac{1}{\phi(t)-\phi(0)}=\frac{1-\frac{\phi^{-1}(\infty)}{t}}{\phi(\infty)-\phi(0)} ,\eeg
\item[iii)] \beg \frac{1}{\phi(t)-\phi(\infty)}=\frac{\frac{t}{\phi^{-1}(\infty)}-1}{\phi(\infty)-\phi(0)}. \eeg

\end{itemize}
\end{lem}

 Proof. Applying  A.1. i) for $\psi (t)=\frac{1}{\phi(t)-p}$ we get  A.2. i). The remaining equalities can be verified by a straightforward calculation.

\qed

{}

\end{document}